\documentclass[11pt,a4paper,twoside,english]{article}
\usepackage{graphics}
\usepackage{color}
\usepackage{textcomp}
\usepackage[pdftex]{graphicx}
\usepackage{comment}
\usepackage{xcolor}
\usepackage{amsfonts}
\usepackage{listings}          
\usepackage{fancyhdr}
\usepackage{indentfirst}
\usepackage{array} 
\usepackage{showkeys}
\usepackage{newlfont}
\usepackage{microtype} 
\usepackage[english]{babel}
\usepackage[applemac]{inputenc}    
\usepackage[T1]{fontenc}
\usepackage{lmodern}
\usepackage{amscd}

\usepackage[fleqn,tbtags]{amsmath}
\usepackage{amssymb,mathrsfs}       
\usepackage{dsfont}                            
\usepackage{mathtools}
\mathtoolsset{showonlyrefs}			
\numberwithin{equation}{section}             

\usepackage{latexsym}
\usepackage{amsthm}
\usepackage{fancyhdr}
\usepackage{verbatim}
\theoremstyle{definition}                          
\newtheorem{thm}{Theorem}[section]     
\newtheorem{prop}[thm]{Proposition}      
\newtheorem{cor}[thm]{Corollary}            
\newtheorem{lem}[thm]{Lemma}             
\newtheorem{hypothesis}[thm]{Hypothesis}
\newtheorem{property}[thm]{Property}
\theoremstyle{definition}               
\newtheorem{dfn}[thm]{Definition}

\newtheorem{rem}[thm]{Remark}          

\newtheoremstyle{prova}
{3pt}
{3pt}
{}
{}
{\textbf}
{.}
{\newline}
{}
\theoremstyle{prova}

\def\R{\mathbb R}
\def\N{\mathbb N}

\def\E{\mathbb E}

\def\P{\mathbb P}

\def\sha{{\cal A}}
\def\shb{{\cal B}}

\def\shf{{\cal F}}

\def\shl{{\cal L}}

\def\shz{{\cal Z}}
\def\1{\mathds{1}}
\def\ra{\rightarrow}

\def\be{\begin{equation}}
\def\ee{\end{equation}}
\def\bs{\begin{split}}
\def\es{\end{split}}

\usepackage{authblk}

\author[1]{Carlo Ciccarella}
\author[2]{Francesco Russo}
\affil[1]{EPFL Lausanne\\
  Institut de Math\'ematiques\\ Station 8\\ CH-1015 Lausanne
  (Switzerland)
}

\affil[2]{ENSTA Paris\\
  Institut Polytechnique de Paris\\
Unit\'e de Math\'ematiques appliqu\'ees\\
F-91120 Palaiseau (France)
}

\begin{document}

\date{September 15th, 2025}

\title{Verification theorem related to a zero sum
    stochastic differential game, based on a chain rule
  for non-smooth functions.}


\maketitle

\begin{abstract}
  In the framework of stochastic zero-sum differential games,
 we establish a verification theorem, inspired by those existing in stochastic
control, to provide sufficient conditions for a pair of feedback controls to form a Nash equilibrium.
Suppose the validity of 
 the classical Isaacs' condition 
and the existence of a (what is termed)   quasi-strong solution to the Bellman-Isaacs (BI)
equations. If the diffusion coefficient of the
state equation is non-degenerate,
we are able to show the existence of a saddle point
constituted by a couple of feedback controls
that achieve the value of the game:
moreover, the latter is equal to the (necessarily unique) solution of the BI equations.
A suitable generalization is available when the diffusion is possibly degenerate.
Similarly we have also improved a well-known verification theorem in
stochastic control theory.
The techniques of stochastic calculus via regularization we use, in particular 
specific chain rules,  are borrowed from a companion paper of the authors.

\end{abstract}

\medskip



\section{Introduction and problem formulation}\label{Intro}

The primary aim of this paper is to establish a verification theorem for a zero-sum stochastic differential game (SDG).
By verification theorem in the context of SDG (resp. stochastic control),
one generally intends a theorem
that provides
  sufficient conditions for a candidate value function, e.g.
  the solution to the Bellman-Isaacs equation (resp. Hamilton-Jacobi-Bellman equation), to be equal to the value of the game
(resp. the optimization value), see e.g.  Definition \ref{Values} (resp. \eqref{eq:ValueContr}).
Additionally, in the case of SDG
 one provides sufficient conditions for a pair of feedback controls to form a Nash equilibrium,
 which corresponds to
 a saddle point for the payoff functional, when the SDG has zero sum.

The game we are interested in, is defined precisely as below.
We will deal with a fixed horizon problem so that we fix $T\in ]0, \infty[$,
a finite dimensional Hilbert space, say $\R^d$ that will be the state space, a finite dimensional Hilbert space, say $\R^m$ (the noise space), two compacts
sets $U_1, U_2 \subseteq \R^k$ (the control spaces).
We consider an initial time and state $(t,x)\in [0,T]\times \R^d$.
Let us fix a stochastic basis
$(\Omega, \shf, (\shf_s) _{s\in [t,T]}, \P) $ satisfying the usual conditions.
$W$ is an $(\shf_s)_{s\in [t,T]}$ $d$-dimensional Brownian motion.

In this work, we adopt a feedback (closed loop) control formulation:
for $i = \{1,2\}$,
$\shz_i(t)$ will be the space of the so called \textit{feedback controls}, that are Borel functions
$z_i:[t,T] \times \R^d \ra U_i$,
corresponding to the controls employed Player $i$.

The state process equation is
\begin{equation}\label{state1}
\left\{
\begin{array}{l l}
dy(s)=  f(s,y(s),z_1(s,y(s)),z_2(s,y(s))) ds + \sigma (s,y(s)) dW_s, \\
 y(t)= x,
\end{array}\right. 
\end{equation}
where the coefficients are defined as
 \begin{align}\label{coef_space}
&f : [0,T] \times \R^d \times U_1 \times U_2 \ra \R^d, \\
&\sigma: [0,T] \times \R^d \ra L(\R^m,\R^d). \nonumber
 \end{align}
$L(\R^m,\R^d)$ will be the space of  $d \times m$ real-valued matrices.
  Equation \eqref{state1} may not have a solution for some
  $z_1, z_2$.
  We state below Hypothesis \ref{hyp:ex-uniq}
  under which  this will not happen and
  Proposition \ref{baseHypFeed} provides tools to verify that hypothesis.
  
  As anticipated earlier, as far as the SDG formulation is concerned,
  we adopt a feedback control formulation.
 While in control theory open loop controls are often employed, in the game theory setting they do not account for the actions of the opponent during the game. As a consequence, they fail to capture the strategic interdependence that characterizes dynamic games.
Feedback controls were employed for SDG for the first time in \cite{FriedGames};
other more recent approaches were performed in the Hamad\`ene-Lepeltier contributions
\cite{Hamadene1} and \cite{Hamadene3}.
An alternative approach is the control versus strategy framework, in which one player selects a control (typically open loop),
 and the other selects a strategy, defined as a state-feedback mapping, see \cite{FS}, \cite{ElKal1972}.
 This formulation  introduces however an asymmetry between players.

 Given $z_i \in \shz_i(t), i = 1,2,$ and a unique strong solution $y$  to \eqref{state1}, when it exists, 
the payoff function may not always be well-defined. If it exists, it is defined by
\begin{equation}\label{J0}
  J(t,x;z_1, z_2)= \E \left \{ \int_t^T l(s,y(s), z_1(s,y(s)), z_2(s,y(s))) ds
    + g(y(T)) \right \},
\end{equation}
where $l$ (resp. $g$) is the running (resp. final) cost.
More precise assumptions on $f$, $\sigma$, $l$,
will be given in Section \ref{Section2}.
Hypothesis \ref{Path0}  provides sufficient conditions for the integral
$$\int_t^T l(s,y(s), z_1(s,y(s)), z_2(s,y(s))) ds,$$
to be always well-defined but, in general, its expectation  may not exist.
Therefore, we introduce the two auxiliary payoff functions
\begin{equation}\label{Jpm1}
J^{\pm} (t,x;z_1, z_2) = \left\{
\begin{array}{l l}
J(t,x;z_1, z_2) \qquad &\text{if well-defined},\\
\pm\infty \qquad &\text{otherwise.}
\end{array}\right.
\end{equation}
There are two notions of value of the game: the upper (resp. lower)
value related to the Player 1 (resp. Player 2) who is supposed to maximize (resp. minimize)
on feedback controls for which the integral and expectation in \eqref{J0} exist,
see Definition \ref{Values} below.

\begin{dfn}\label{Values}
  \begin{enumerate}
\item The \textbf{upper value} $V^+$ and \textbf{lower value} $V^-$ of the game (SDG) \eqref{state1} and \eqref{J0} with initial data $(t,x)$ are given by
\begin{align*}
  V^+(t,x) &= \inf_{z_2 \in \shz_2(t)} \sup_{z_1 \in \shz_1(t)} J^+(t,x; z_1, z_2 )
  \\
V^-(t,x) &=  \sup_{z_1 \in \shz_1(t)} \inf_{z_2 \in \shz_2(t)} J^-(t,x; z_1, z_2 ).
\end{align*}
\item If $V^+(t,x) = V^-(t,x)$, then we say that the SDG \eqref{state1} and \eqref{J0} has a \textit{value} and we call 
$V(t,x) = V^+(t,x) = V^-(t,x)$ the \textbf{value of the game}.
\end{enumerate}
\end{dfn}
\begin{rem}\label{RInfSup}
  Since $\sup \inf \le \inf \sup$ and
   $J^-(t,x; z_1, z_2 ) \le  J^+(t,x; z_1, z_2 ),$ 
  we observe that, in general $V^-(t,x) \le V^+(t,x)$.
\end{rem}

The next definition is essentially the definition of a Nash equilibrium for the game.
For nonzero-sum games, a notion of Nash equilibrium 
has been introduced for instance
in  Section 1 in \cite{Hamadene3} or (2.8) in \cite{FriedGames}.
\begin{dfn}\label{Saddle}
  Let $t \in [0,T]$. A couple of controls $(z^\star_1,z^\star_2) \in \shz_1(t) \times \shz_2(t)$ for which, for all $s\in [t,T]$, $x\in \R^d$,
  $J(t,x; z^\star_1, z^\star_2)$ is well-defined, finite and
  \begin{equation}
    \label{saddleClass}
 J(t,x; z_1, z^\star_2) \leq J(t,x; z^\star_1, z^\star_2) \leq J(t,x; z^\star_1, z_2),
     \end{equation}
     for every  $z_i \in \shz_i(t), i = 1,2$.
In particular $J(t,x; z_1, z^\star_2) $ and  $J(t,x; z_1^\star, z_2) $ are well-defined.
     \end{dfn}
       Previous definition can be formulated equivalently as follows.
     \begin{rem} \label{RSaddle}
      $(z^\star_1,z^\star_2)$ is a saddle point if for every  $z_i \in \shz_i(t), i = 1,2$, we have
     \begin{align}
      & J^+(t,x; z^\star_1, z^\star_2) = J^-(t,x; z^\star_1, z^\star_2) = J(t,x; z^\star_1, z^\star_2),
      \\
      &  J^-(t,x; z_1, z^\star_2) \leq J(t,x; z^\star_1, z^\star_2) \leq J^-(t,x; z^\star_1, z_2), \\
&  J^+(t,x; z_1, z^\star_2) \leq J(t,x; z^\star_1, z^\star_2) \leq J^+(t,x; z^\star_1, z_2).
  \end{align}
  Indeed, if for some  $z_1$  (resp.  $z_2$), $J(t,x; z_1, z^\star_2)$ (resp. $J(t,x; z^\star_1, z_2)$)
  is not defined, then $(z^\star_1,z^\star_2)$ cannot be a saddle point.
\end{rem}

If there exists a saddle point $(z^\star_1,z^\star_2)$, then Proposition \ref{prop:verification}
implies that the game has a value. 
On the other hand, if the game has a value, this does not guarantee that there exists some $(z^\star_1,z^\star_2)$, which actually achieve it.
As we will explain later, in this work we provide natural sufficient conditions for the existence of such a couple.

The motivation behind the upper and lower value functions stems from the inherent ambiguity in defining the value of the game. This ambiguity arises  since the supremum (resp. infimum) is calculated prior to the infimum (resp. supremum) of the payoff function, denoted as \eqref{J0}. Such a process typically yields two distinct outcomes, precisely $V^+$ (resp. $V^-$). The concepts of upper and lower values were initially introduced in deterministic frameworks, as detailed in the works \cite{FlemingDet1}, \cite{FlemingDet2}, \cite{ElliottKalDet1}, and \cite{SST}.

In the literature, the PDEs  playing a similar role  to the Hamilton-Jacobi-Bellman equation (HJB)
in case of control theory, are the upper and lower Bellman-Isaacs equation (BI): they are defined below  in \eqref{HJB1} and \eqref{HJB2}. These equations are associated with the stochastic differential game \eqref{state1} and \eqref{J0} and they were first formally derived 
 in \cite{Isaacs}.
To express them, one first defines the current value Hamiltonian  $H_{CV}: [0,T] \times \R^d \times \R^d \times U_1 \times U_2 \ra \R $ as
\begin{equation} \label{Hext}
H_{CV}(s,x,p,u_1,u_2) = \langle f(s,x,u_1,u_2),p\rangle + l(s,x,u_1,u_2).
\end{equation}
Setting formally
\begin{align}\label{shl0}
\shl v(s,x) = \partial_s v(s,x) + \frac{1}{2} Tr [\sigma^\top(s,x) \partial_{xx}v(s,x) \sigma(s,x)],
\end{align}
we can write the Bellman-Isaacs equations associated with the problem \eqref{state1} and \eqref{J0} as
\begin{equation}\label{HJB1}
\left\{
\begin{array}{ll}
\shl v (s,x) + H^{+} (s,x,\partial_x v(s,x)) =0,\\
v(T,x)= g(x),
\end{array}\right.
\end{equation}
\begin{equation}\label{HJB2}
\left\{
\begin{array}{l l}
 \shl v (s,x) + H^{-} (s,x,\partial_x v(s,x)) =0,\\
v(T,x)= g(x),
\end{array}\right.
\end{equation}
where  $H^{-} , H^{+} : [0,T] \times \R^d \times \R^d
\ra \R $
are the so called {\it Hamiltonians} defined by
\begin{equation}\label{H-}
H^{-}(s,x,p)  = \sup_{u_1 \in U_1} \inf_{u_2 \in U_2} H_{CV}(s,x,p,u_1,u_2)  , 
\end{equation}
\begin{equation}\label{H+}
H^{+}(s,x,p)  = \inf_{u_2 \in U_2} \sup_{u_1 \in U_1}  H_{CV}(s,x,p,u_1,u_2) .
\end{equation} 
Equation \eqref{HJB1} (resp. \eqref{HJB2}) is usually called the \textit{upper} (resp.
\textit{lower})
\textit{Bellman-Isaacs (BI) equation}.

Many results proving the existence of a value of the game rely on a
so-called {\it Isaacs' condition}.
Among those,  \cite{FriedGames}, \cite{Car3},  \cite{Hamadene3} consider non zero-sum games,  while
\cite{FS}, \cite{Fleming2011}, \cite{Buckdahn2008}, \cite{CarRain2008}, \cite{BayYao2012}, \cite{Elliott1976}, \cite{Elliott1977}, \cite{Hamadene1}, \cite{KarHam2003}, \cite{PhaZha2018} are in the setting of  zero-sum games. Some important contributions which prove  existence of a value in a zero-sum game without assuming the Isaacs' condition are \cite{subbotin}, \cite{quincampoix} and \cite{sirbu}.
The usual formulation of the Isaacs' condition in the case of a zero-sum game is the one of
Definition \ref{Isaacs0} below.

\begin{dfn}\label{Isaacs0}
  The SDG \eqref{state1} and \eqref{J0} is said to fulfill the {\bf Isaacs' condition} if $H^{+}(s,x,p) =H^{-}(s,x,p)$
  for all $s \in [0,T]$, $x \in \R^d$ , $p \in \R^d$.
\end{dfn}
If Isaacs' condition holds, the couple of equations \eqref{HJB1} and \eqref{HJB2} collapse into only one equation.

There are four main approaches to zero-sum stochastic differential games, some of them mentioned earlier.

The first approach is concerned by verification theorems via a PDE-based approach, employing feedback controls, see
\cite{FriedGames},
which focuses on classical $C^{1,2}$-solutions to the Bellman-Isaacs equation and presents an approach closely aligned with ours.
It is proved that, if
the  Isaacs' condition is fulfilled,
then there is a pair of feedback controls
which  constitute  a saddle point of the game, as defined  in \eqref{saddleClass}.
   Furthermore, the paper extends the results to $N$-player games and games with incomplete information.

In the second one, a game is formulated in a control against strategy setting,
in the sense that player acting first chooses a control and
the other responds with a strategy.
The player choosing a strategy has information also on the control used by his opponent.
In that context, the definition of lower (resp. upper) value is similar to
the one in Definition \ref{Values},
where the supremum (resp. infimum) over controls is replaced by a supremum (resp. infimum) over
strategies.
This formulation was  introduced in the deterministic framework in
\cite{ElKal1972}.
The earliest work which establishes
the existence of a value in a stochastic differential game using this formulation appears to be \cite{FS}. Subsequently, the distinction between controls and strategies (appearing in \cite{FS}) is mitigated in \cite{Fleming2011}.
In \cite{FS} and \cite{Fleming2011}, it is proved that 
the upper (resp. lower) value function (see Definition \ref{Values}) is a viscosity solution to the upper (resp. lower) Bellman-Isaacs equation, see \eqref{HJB1} (resp. \eqref{HJB2}).
In particular, if the  Isaacs' condition holds, see Definition \ref{Isaacs0},
then the game has a value, in the sense that lower and upper value coincide,
 whenever those equations are well-posed.
 In \cite{Buckdahn2008}, the authors generalize the results of \cite{FS} expressing the cost functional through the solutions of a backward SDE (BSDE) in the sense of Pardoux and Peng \cite{pardouxpeng}.
The authors also prove a dynamic programming principle for the upper and lower value functions of the game
 in a direct manner, without relying on the approximation argument used in \cite{FS}.
 In all the aforementioned contributions of this
 second vague, no verification theorem is used
 and the existence of the value of the game can be proved as viscosity solution
of a corresponding Bellman-Isaacs equations.

 The third stream of literature also  deals with  stochastic differential games in the
 setting ``strategy against control'',
but the formulation is weak, non-Markovian and
  the set of admissible controls and strategies are of feedback type.
  In \cite{DavisVaraiya1972}, \cite{Elliott1976}, \cite{Elliott1977} and  in \cite{DavisElliott1981} necessary and sufficient conditions are provided for the existence of a value.
  They use a principle of optimality for non-Markovian controlled processes proved in  \cite{DavisVaraiya1972}, Theorem 4.1. In this context, a verification theorem appears in \cite{DavisVaraiya1972},
  Theorem 7.2.
  
  The fourth (more recent) approach makes use, similarly to the present paper, of a setting ''control against control'' and is employed in  \cite{Hamadene1}, \cite{KarHam2003}, \cite{Hamadene3}, \cite{PhaZha2018}. The authors formulate and prove a verification theorem via
  BSDEs techniques, showing that a specific couple of feedback controls  $(z_1^\star,z_2^\star)$, verifying some inequality of saddle point type, constitute a Nash equilibrium for SDG. 

  We repeat that, in the second and third approach, the games are asymmetrically formulated, namely, the player who plays last, has information also  on the control used by her/his opponent.
  In the first and fourth approaches both players implement a
  ``control against control'' approach and they play simultaneously: nevertheless, the second and fourth formulations are reconciled in \cite{Rain4}.

Besides the Friedman's seminal paper \cite{FriedGames}, we are not aware of any contributions proving a verification theorem using a PDE argument
in the classical framework of SDG. Our work makes a  contribution to the literature by presenting a verification theorem that employs a PDE approach while assuming non-classical regularity ($C^{0,1}$) for the solution to the Bellman-Isaacs PDE.  Specifically, we suppose the existence of a (what we call) quasi-strong solution (see Definition \ref{strong1})  $v$ of the upper/lower Bellman-Isaacs equations and the validity of
  the Isaacs  condition, see Definition \ref{Isaacs0}.
  That notion of quasi-strong solution was introduced in \cite{CR1}, Section 3.3.
Under previous conditions, if there exist some maps $(z^\star_1, z^\star_2)$ fulfilling the measurable selection Hypothesis \ref{hyp_select},
then $v(t,x)$  is equal to value of the SDG (see Definition \ref{Values})
and
$(s,y) \mapsto (z^\star_1(s,y,\partial_x v(s,y), z^\star_2(s,y,\partial_x v(s,y))))$ is a saddle point,
see Definition \ref{Saddle}.

Other works, to the best of our knowledge, proving a verification theorem using other techniques, are  the already mentioned 
papers \cite{Hamadene1}, \cite{Hamadene3},
\cite{DavisVaraiya1972} together with \cite{HuangShi}, the latter in the different setting of Stackelberg games.
Moreover, verification theorems, in the different context of impulse games, are provided by \cite{AidBasei}, \cite{Campi} and references therein.  In these works, classical $C^2$-regularity of the solution to the Bellman-Isaacs equation is assumed. 

Concerning the hypotheses in our paper,
we insist first on the fact that
the Hamiltonian \eqref{Hext}, all the coefficients of the state equation \eqref{state1}, the running cost and the BI equations \eqref{HJB1} and \eqref{HJB2} are only supposed to be continuous with respect to the space variable and neither with respect to the time nor with respect to the control variables. This explains the motivation of working with quasi-strong solution of BI equations.
In particular,  we allow  changes of regimes in the dynamics of the state equation.
Indeed we are not aware of contributions in the literature  dealing with strong formulation of the differential games which do not require time continuity of the coefficients.

Our work significantly weakens the assumptions of \cite{FriedGames},
where the diffusion is assumed to be non-degenerate and continuous in
$s$, together with its spatial derivative. Moreover, in \cite{FriedGames} also the drift and the running cost, $l$ in \eqref{J0}, should be continuous with respect to all the entries.  We also overcome non-degeneracy with a condition of convergence on the derivatives of the value function, see Hypothesis \ref{3.7}, item 1. 
 
Second, the verification theorems proved in \cite{Hamadene1} and \cite{Hamadene3}   require  continuity on the dependence of the saddle point controls on the state and the derivative of the value function, a condition which does not intervene in our case, see Remark \ref{RVerifHamadene}.
 Additionally, they also assume non-degeneracy of the diffusion.


 Finally, we do not assume a priori the existence of the payoff functional $J$, defined in  \eqref{J0}, for our class of admissible controls.
 We will prove that, for the saddle point couple
 $(z_1^\star, z_2^\star)$,
 $J$ is automatically well-defined and integrable.

 Section  \ref{S5}  is devoted to stochastic control, where
 we formulate a slight generalization of (the verification) 
 Theorem 4.9 of \cite{rg1},
 making use of a quasi-strong solution of an Hamilton-Jacobi-Bellman (HJB) equation. 
The context of our control problem is different from the one of SDG, since the  set of controls is constituted by (open loop)
progressively measurable stochastic processes and not functions representing feedback controls.
Here again, all the coefficients of the state equation \eqref{stateC}, of the payoff functional \eqref{TJC}, and of the HJB equation \eqref{HJBC} are
only supposed to be continuous with respect to the space variable and not with respect to the time and control variables.
Remark that  in \cite{rg1} they were required to be continuous with respect to all the entries. 
Also, in  Theorem 4.9 and Lemma 4.10 of \cite{rg1} the authors supposed the lower integrability condition in the running cost $l$, i.e.
  \begin{equation*}
 \E \left( \int_t^T  l^-(s,y(s,t,x,z), z(s)) ds \right) > - \infty,
 \end{equation*}
  where $l^- \doteqdot (l \wedge 0 )$.
This is not anymore required, see Hypothesis \ref{Path0}.
Finally the terminal cost $g$  is not required  anymore to be differentiable.
We suppose the existence of quasi-strong solutions of the HJB equations \eqref{HJBC}.
That notion is weaker than the corresponding concept of strong solution used in \cite{rg1} and \cite{rg2}.

Our applications to game theory and stochastic control were made possible via generalization of some stochastic calculus via regularization tools developed in \cite{CR1}, in particular from that paper, we recall below Theorem \ref{representation},
based on a Fukushima-Dirichlet decomposition.

Our results are organized as follows. 
 
In Section \ref{sec:prel} we fix some notations and give some stochastic calculus preliminaries. We recall the definition of a quasi-strong and quasi-strict solution and we state
a chain rule for $C^{0,1}$-quasi-strong solutions of PDEs, i.e.
Theorem \ref{representation}, which was essentially Corollary 4.1 in \cite{CR1}.
In the whole Section \ref{Chapter3} we switch to the zero-sum stochastic game theory setting. In Section \ref{Section2} we state the main hypotheses for the state equation \eqref{state1} and the integrand of the payoff \eqref{J0} to be well-defined.
Section \ref{S32} proves the so called {\it fundamental lemma} (see Lemma \ref{FundamLemma1}) that is a
consequence of the chain rule Theorem \ref{representation}. It is used in
Section \ref{S33}, which is the core of the paper and where it is proved a verification theorem for a zero-sum stochastic differential game. The same theorem shows that the game has a value.
Section  \ref{S5}  is devoted to the stochastic control theory improving some of the results in \cite{rg1}.

\section{Preliminaries and Stochastic calculus}

\label{sec:prel}

First, we recall some basic definitions, we fix some notations,
in particular the notions of quasi-strong solution
of a parabolic PDE 
and we recall the It\^o type chain rule 
that plays a key role in
the proof of Theorem
\ref{Verification3}. 


\subsection{Preliminaries}

 \label{SPrelim}

In this section, $ 0 \leq t < T <\infty$ will be fixed. The definition and conventions of this
section will be in force for the whole paper.


If $E$ is a (finite-dimensional) linear space,
 $C^{0} (E) $ will denote the space of all continuous functions $f: E \rightarrow \R$.
For $k$ be a positive integer,
$C^k(\R^d)$ denoted the space of the functions such that all the derivatives up to order $k$ exist and are continuous.
Let $I$ be a compact real interval.   $C^{1,2} ( I \times \R^d)$ (respectively $C^{0,1} ( I \times \R^d)$), is the space of continuous functions
$f\ : \ I \times \R^d \ra \R $,  
 $ \ \ \ (s,x) \mapsto f(s,x)$,
such that  $ \partial_s f, \partial_x f, \partial^2_{xx} f  $ (respectively $ \partial_x f $)
are well-defined  and  continuous.
In general $\R^d$-elements will be considered as column vector,
with the exception of $\partial_x f$ which will be by default
a row vector.

We introduce now a useful space of functions.
\begin{dfn}\label{C02ac}
  $C^{0,2}_{ac}(I \times \R^d) $ will be  the linear space of continuous
  functions $f: I \times \R^d \ra \R  $ such that
  the following holds.
\begin{enumerate}
 \item $f(s,\cdot) \in C^2 (\R^d)$ for   all $s \in I$.
 \item For any $x \in \R^d $ the function $s\mapsto f(s,x) $ is absolutely continuous and $x \mapsto \partial_s f(s,x)$ is continuous for almost all $s \in I$.
\item  For any compact  $K \subset \R^d $,  $\sup_{x\in K} |\partial_s f(\cdot, x) | \in L^1 (I).$
\item
  Let $g$ be any second order space derivative of $f$.
  \begin{enumerate}
\item  $g$ is continuous with respect to the space variable $x$ varying on each compact $K$, uniformly with respect to $s \in I;$
  \item  {\it for every $x \in \R^d$ , $s \mapsto g(s,x)$ is a.e. continuous.}
\end{enumerate}
\end{enumerate}
\end{dfn}

\subsection{Concept of PDE solution}

First of all, we introduce a parabolic partial differential equation that will be used in all the work. 
Until the end of this section we suppose
$\sigma:[0,T] \times \R^m \rightarrow L(\R^m, \R^d) $  to be
a locally bounded Borel function, i.e.
for all $K$ compact, $\sup_{r \in [0,T], x\in K} |\sigma(r,x)| < \infty$.
 We recall the formal definition of the linear parabolic operator $\shl$ in \eqref{shl0}.
In this section we will fix $t \in [0,T[$ and Borel functions
$g: \R^d \rightarrow \R$
and $h :[t,T] \times \R^d \ra \R $ such that
\begin{equation} \label{hCond}
  \int_t^T \sup_{x \in K} |h(s,x)| ds < \infty,
\end{equation}
  for every compact $K \subset \R^d$.
We will   consider the inhomogeneous backward parabolic problem
\begin{equation} \label{parabolic}
 \left\{
\begin{array}{l l}
\shl u(s,x)= h(s,x), & s\in [t,T], \ x\in \R^d,\\
u(T,x)= g(x), & x\in \R^d.
\end{array}\right. 
\end{equation}
The  definition below generalizes the notion of strict solution, that
appears in \cite{rg1}, Definition 4.1.
For $s \in [t,T],$  we set
\begin{equation} \label{eq:As}
\sha_s f(x) \doteqdot  \frac{1}{2} Tr [\sigma^\top(s,x) \partial_{xx}f(x) \sigma(s,x)].
\end{equation}
\begin{dfn}\label{strict}
 We say that
$u : [t,T] \times \R^d \ra \R$, $ u \in C^{0,2}_{ac}([t,T]\times \R^d )$ is a \textit{quasi-strict solution} to the backward Cauchy problem \eqref{parabolic} if 
\begin{equation} \label{Estrict}
  u(s,x)= g(x) -  \int_s^T h(r,x) dr + \int_s^T (\sha_r u)(r,x) dr, \forall s \in [t,T],  x \in \R^d.
\end{equation}
 
\end{dfn}
\begin{rem}  \label{Rstrict}
  In this case, for every $x \in \R^d$,
 \begin{equation}\label{Estrict1}
  \partial_s u(s,x) =  h(s,x) - (\sha_s u)(s,x) \qquad a.e.,
  \end{equation}
  where $\partial_s u$  stands for the distributional derivative of $u$.
  \end{rem}

The notion of quasi-strict solution allows to consider
(somehow classical) solutions of \eqref{parabolic}
even though $h, \sigma$ are not continuous in time.

The definition below is a relaxation of the notion
of strong solution defined for instance in \cite{rg2}, Definition 4.2.,
which is based on approximation of strict (classical) solutions.


\begin{dfn}\label{strongNU}
  $u \in C^0 ([t,T]\times \R^d)$ is a \textit{quasi-strong solution}
(with  \textit{approximating sequence} $(u_n)$)
to the backward Cauchy problem \eqref{parabolic} if there exists
a sequence  $(u_n) \in C^{0,2}_{ac}([t,T] \times \R^d)$ and two sequences
$(g_n): \R^d \rightarrow \R$,
$(h_n):[t,T] \times \R^d \rightarrow \R$
such that 
$ \int_t^T \sup_{x \in K} |h_n(s,x)| ds < \infty,  $ for every compact $K \subset \R^d$ 
realizing the following.
\begin{enumerate}
\item $\forall n \in \N$, $u_n$ is a quasi-strict solution of the problem \\
\begin{equation}  
\left\{
\begin{array}{l l}
\shl u_n(s,x)= h_n(s,x) &s \in [t,T], \ x\in \R^d,  \\
u_n(T,x)= g_n(x) & x\in \R^d. 
\end{array}\right. 
\end{equation}

\item For each compact $K \subset \R^d $ 
 \begin{equation}
   \left\{
 \begin{array}{l l}
  \sup_{(s,x) \in [t,T] \times K } | u_n - u |(s,x)  \ra 0,             \\ 
\sup_{x \in K} |h_n -h|(\cdot,x) \ra 0  \qquad \text{in } L^1([t,T]). 
\end{array}\right. 
\end{equation}
\end{enumerate}
\end{dfn}

\begin{dfn} \label{Non-deg}
We say that  $\sigma$ is {\bf non-degenerate}
  if  there is a constant $c > 0$ such that for
  all $(s,x) \in [0,T] \times \R^d$ and $\xi \in \R^d$ we have
\begin{equation} \label{not-deg}
  \ \xi^\top \sigma(s,x)^\top \sigma(s,x) \xi \ge c \vert \xi\vert^2.
\end{equation}
$\sigma$ will be said {\bf degenerate} if it is not non-degenerated.
\end{dfn}

\begin{rem} \label{MildStrong}
  The notion of quasi-strong solution is natural and it can be
  related to the notion of mild solution, 
  see Section 3.5
  of \cite{CR1}.
\end{rem}

\subsection{The It\^o chain rule formula}

In the sequel $(\Omega, \shf, (\shf_s)_{s\geq t}, \P)$ will be a given stochastic basis satisfying the usual conditions. $(W_s)_{s\geq t} $ will denote a classical
$(\shf_s)_{s\geq t} $-Brownian motion with values in $\R^d$. 
A sequence of processes $(X^n_s)$ will be said to converge u.c.p.
if the convergence holds in probability uniformly on compact intervals.

The following result is a slight straightforward generalization of Corollary 4.1
in \cite{CR1},
 that we formulate here only when $ b = 0$.
When $u$
is a strong solution,
it
follows from Corollary 4.6 of \cite{rg2}.
We drive the attention that in the present case
the coefficients $\sigma$ is not continuous,
in particular with respect to time, and we will focus
on the case when $u$ is a quasi-strong solution.

Here, we also assume $u \in C^{0,1}([t,T[\times \R^d) \cap C^{0}([t,T] \times \R^d)$, that is,
$u(T,\cdot)$ may not be differentiable in the space variable.
This allows lower regularity in the terminal condition of the
PDE \eqref{parabolic}.

\begin{thm}\label{representation}
Let $\sigma : [0,T]\times \R^d \rightarrow L(\R^m , \R^d) $ be a Borel locally bounded function
  and $F_0: \Omega \times [0,T]  \rightarrow \R^d$ be an a.s. locally bounded
  progressively measurable field. 
  Let $(S_s)_{s\in [t,T]}$ be a It\^o process such that 
$$
dS_s = F(s) ds + \sigma(s,S_s)dW_s
$$
and
  $$ F(\omega,s) = \sigma \sigma^\top(s,S_s) F_0(\omega,s).$$
Let
 $h:[t,T] \times \R^d \rightarrow \R$ and $g: \R^d \rightarrow \R$
as in the lines before \eqref{hCond} and above.
Let  $u  \in C^{0,1}([t,T[\times \R^d) \cap C^{0}([t,T] \times \R^d)$ be a quasi-strong solution of the
Cauchy problem \eqref{parabolic}
  fulfilling
\begin{equation}\label{L21} 
\int_t^T \vert \partial_x u(s,S_s) \vert^2 ds  < \infty  \ {\textit a.s.}
\end{equation}
Moreover we suppose the validity of one of the following items.
\begin{enumerate} 
\item The approximating sequence $(u_n)$ of Definition \ref{strongNU}
  fulfills
  \begin{align}\label{16}
 \lim_{n\rightarrow \infty}  \int_t^\cdot  (\partial_x u_n(r,S_r)- 
\partial_x u(r,S_r))  F(r) dr = 0 \ \ \text{u.c.p.}
\end{align}
\item
   The Novikov condition  
\begin{equation} \label{Novikov}
  \E\left ( \exp\left(\frac{1}{2} \int_t^T
 \vert \sigma^\top(r,S_r) F_0(r)\vert^2 dr\right)\right) < \infty,
  \end{equation}
  holds.
\end{enumerate}
Then, for $s\in [t,T],$
\begin{equation}\label{215}
  u(s,S_s)= u(t,S_t)  +  \int_t^s  \partial_x u (r,S_r)\sigma (r,S_r) dW_r +
  \shb^S(u)_s,
\end{equation}
where,
\begin{equation}\label{216}
  \shb^S(u)_s   = \int_t^s  h(r,S_r)dr + \int_t^s  \partial_x u(r,S_r)  F(r)  dr.
  \end{equation}
\end{thm}
\begin{rem}
  \begin{enumerate}
    \item
Item 2. of Corollary 4.1 
in \cite{CR1} was formulated under the assumption that $\sigma \sigma^\top$ is invertible.
The condition \eqref{Novikov} was expressed saying that,
setting $ \sigma^{-1}:= \sigma^\top (\sigma \sigma^\top)^{-1},$
$ \sigma^{-1} F$ verifies the Novikov condition.
Clearly, whenever $(\omega,s) \mapsto \sigma^{-1}(s,S_s) F(\omega,s)  $ \\
is bounded, then \eqref{Novikov} is fulfilled.
The present formulation is a generalization and the adaptation
of the proof is straightforward.
\item
  If $\lim_{n\rightarrow \infty} \partial_ xu_ n = \partial_x u $   in $C^0 ([t,T]\times \R^d)$, then assumption \eqref{16} is trivially verified.
  \end{enumerate}
\end{rem}

\section{Application to game theory}\label{Chapter3}

In this section we prove the verification theorem. It states a sufficient condition for a pair of feedback controls $(z_1,z_2)$, for a given $t\in [0,T[$, to constitute a {\it Nash equilibrium}, i.e. a saddle point (see Definition \ref{Saddle}) 
in a zero-sum stochastic differential game, as formulated in Theorem \ref{Verification3}, item 3.

\subsection{Basic setting}\label{Section2}

We give first precise assumptions on the coefficients.

\begin{hypothesis}\label{3.1}
The functions $f, \sigma$, introduced in \eqref{coef_space}, fulfill the following.

\begin{enumerate}
\item $f, \sigma$ are  Borel functions and $\sigma$ is Lipschitz.
   \item There is a constant  $K > 0$ such that $\forall x, x_1, x_2 \in \R^d, t \in [0,T]$, $\forall (u_1,u_2) \in U_1 \times U_2$,
we have 
$$\vert f(t,x,u_1,u_2) \vert + \Vert \sigma(t,x) \Vert  \leq K (1+|x|).$$
    
\end{enumerate}      
\end{hypothesis}
Let us consider $(t,x) \in [0,T] \times \R^d$,
and for a moment $z_1 \in \shz_{1}(t)$ and $z_2 \in \shz_{2}(t).$
Suppose the existence of a unique strong solution (up to indistinguishability) to the state equation \eqref{state1}
In that case, the aforementioned solution
will be denoted by
\begin{equation} \label{ystsol}
  y(s;t,x,z_1, z_2)   \ \text{or} \  y(s), \ s \in [t,T],
  x \in \R^d, z_1 \in \shz_1(t), z_2 \in \shz_2(t).
  \end{equation}
  To establish the  proof of our verification theorem, we also need
  the solution of   \eqref{state1} to have polynomial growth,
i.e. to verify the  Property below.
\begin{property}\label{moments3}
    For $\forall p \ge 1 $,
  there is a constant $N = N(p) > 0$ such that
\begin{equation} \label{eq:moments}
 \E \sup_{t\leq s \leq T}|y(s;t,x,z_1, z_2)|^p  \leq Ne^{NT} (1+ |x|^p).
\end{equation}
\end{property}
The proposition below
is a direct consequence of Theorem 4.6 in \cite{KR}.  
\begin{prop} \label{P44} 
  Let $t \in [0,T[$ and consider
  $ z_1 \in \shz_{1}(t)$ and $ z_2 \in \shz_{2}(t)$.
  Suppose the validity of Hypothesis \ref{3.1}
  and the existence of a  solution $y$ to \eqref{state1}.
Then $y$ fulfills Property \ref{moments3}.

\end{prop}
 
\begin{prop}\label{baseHypFeed}
Assume Hypothesis \ref{3.1}.
     Let $ z_1:[0,T]\times \R^d \mapsto U_1$ and
  $ z_2:[0,T]\times \R^d \mapsto U_2$ be two Borel functions.

  Then a unique strong solution to \eqref{state1}
  exists, which moreover fulfills Property \ref{moments3}, in the two following cases.
  \begin{enumerate}
  \item The non-degeneracy condition \eqref{not-deg} holds.
  \item  $x \mapsto    f(s,x, z_1(s,x),  z_2(s,x))$  is uniformly Lipschitz.
  \end{enumerate}
\end{prop}

\begin{proof}[Proof of Proposition \ref{baseHypFeed}]

  Concerning 1., the SDE \eqref{state1}, whose unknown is $ y$, can be considered as an SDE with Lipschitz non-degenerate diffusion
  coefficient $\sigma$ and linear growth measurable drift
\begin{equation*}
(s,x)  \mapsto f\left(s,x, z_1(s, x), z_2(s, x)\right).
\end{equation*} 
By \cite{veretennikov1982}, Theorem 6., there exists a unique strong solution 
to \eqref{state1}.

Case 2. follows immediately by the standard theorem of existence
and uniqueness for SDEs with Lipschitz coefficients.

 Property \ref{moments3} is a consequence of Proposition \ref{P44}.
\end{proof}







For the statement of the main theorems of the paper
we will need to suppose the validity of the following.

\begin{hypothesis} \label{hyp:ex-uniq}
There exists a unique strong solution (up to indistinguishability) to the state equation \eqref{state1}
for every $(t,x) \in [0,T] \times \R^d$,
for any controls $z_1 \in \shz_{1}(t)$ and $z_2 \in \shz_{2}(t).$
\end{hypothesis}

\begin{rem} \label{rmk:Veret}
  When $\sigma$ is non-degenerate, under Hypothesis
  \ref{3.1}, Proposition \ref{baseHypFeed}
  will state that Hypothesis \ref{hyp:ex-uniq}
  is always fulfilled.

\end{rem}

For the expected payoff \eqref{J0} to be well-defined, we need the validity of Hypothesis \ref{Path0},
 related to the current and final costs $l$ and $g$,
which will presuppose the validity of Hypothesis \ref{hyp:ex-uniq}.

\begin{hypothesis}\label{Path0}
  Let $l: [0,T]\times \R^d \times U_1 \times U_2 \rightarrow \R $ be
  and $g : \R^d \rightarrow \R $ Borel functions.
  Moreover, for any
$t \in [0,T]$ we assume
  that, for all $z_1 \in \shz_{1}(t)$, $z_2 \in \shz_2(t)$, the function $s \mapsto l(s, y(s), z_1(s,y(s)), z_2(s,y(s)))$ is integrable in $[t,T]$, $\omega$ a.s. In particular $$ -\infty  < \int_t^T l(s, y(s), z_1(s,y(s)), z_2(s,y(s))) ds  < \infty \qquad a.s.$$ 
\end{hypothesis}
Previous Hypothesis is indeed quite weak: 
for instance,  if $l$ is locally bounded it is trivially verified.

\begin{dfn} \label{D14} 
  Let us suppose
  Hypotheses
  \ref{hyp:ex-uniq} and \ref{Path0}
  and let $t \in [0,T[$. For a $z_1 \in \shz_{1}(t)$ and a $z_2 \in \shz_2(t)$, we define
  $\tilde{J} : \R^d  \times \shz_1(t) \times \shz_2(t) \times \Omega \ra \R$ by
\begin{equation}\label{TJ0}
\tilde{J}(t,x;z_1,z_2) \doteqdot \int_t^T l(s,y(s), z_1(s,y(s)), z_2(s,y(s))) ds + g (y(T)). 
\end{equation}
\end{dfn}
At this point, the functional $\tilde{J}$ is well-defined for any $(t,x) \in [0,T] \times \R^d$ and $(z_1, z_2) \in \shz_1(t) \times \shz_2(t)$, but not necessarily its expectation $J$.

\subsection{The fundamental lemma}\label{S32}

We now define the notion of quasi-strong and quasi-strict solution of equations \eqref{HJB1} and \eqref{HJB2}.
This concept of solution is necessary to prove the so called fundamental Lemma \ref{FundamLemma1} below. We first introduce an elementary hypothesis and notation, which will be in force
for the rest of Section \ref{Chapter3}.

\begin{hypothesis} \label{hyp:BorelH}
The function $H = H^{-}$ (resp. $H^{+}$),  defined in \eqref{H-} (resp. \eqref{H+}), is
supposed to be Borel.
\end{hypothesis}

\begin{rem} \label{rmk:HamBorel}
  The assumption of Borel measurability of $H^{-}$ and $H^{+}$,
  stated above
  is not necessary if $H_{CV}(s,x,p,u_1,u_2)$ is continuous with
  respect to $(u_1,u_2)$, since it is automatically fulfilled.
  \end{rem}

  \begin{dfn}\label{strict1}
    Let $H = H^{-}$ (resp. $H^{+}$) be the function defined in \eqref{H-} (resp. \eqref{H+}) such that Hypothesis \ref{hyp:BorelH}.
    Let $v \in C^{0,2}_{ac}([0,T], \R^d)$. We set $h(r,x) =  - H(r,x,\partial_x v(r, x)) $, $r \in [0,T]$, $x \in \R^d$. 
We say that $v$ is a \textit{quasi-strict solution} of \eqref{HJB1} (resp. \eqref{HJB2}) if the following holds.
\begin{enumerate}
\item $\int_0^T \sup_{x \in K} |h(r,x)| dr  < \infty$ for every compact $K \in \R^d$,
\item $u = v$ is a quasi-strict solution of \eqref{parabolic}  with $t= 0$. 
\end{enumerate}
\end{dfn}
\begin{dfn}\label{strong1}
  Let $v \in  C^{0,1}([0,T[ \times \R^d) \cap C^{0}([0,T] \times \R^d)$. We set again \\
  $h(r,x) = - H(r,x,\partial_x v(r, x))  $,  $r \in [0,T]$, $x \in \R^d$. 
 We say that $v$ is a \textit{quasi-strong solution} of \eqref{HJB1} (resp. \eqref{HJB2}) if
 $u:= v$ is a quasi-strong solution of \eqref{parabolic} with $t= 0$. 
\end{dfn}

Given a function $v:[0,T] \times \R^d \rightarrow \R^d$,
we will consider the following assumption,
related to
the PDE \eqref{HJB1} (resp. \eqref{HJB2}).

\begin{hypothesis}\label{3.7}
  $v$
  belongs to $ C^{0,1}([0,T[ \times \R^d) \cap C^{0}([0,T] \times \R^d)$ and is a quasi-strong solution  of the Bellman-Isaacs equation \eqref{HJB1} (resp. \eqref{HJB2}) with approximating sequences $(v_n)$.

Moreover, $\partial_x v$ has polynomial growth and one of the two following conditions holds true.
\begin{enumerate}
\item 
The sequence of $(\partial_x v_n)$ converge uniformly on compact sets to $\partial_x v$.
\item
  There is a Borel function $f_0$ such that
$f(r,x,u_1,u_2) = \sigma \sigma^\top f_0(r,x,u_1,u_2)$ and
 \begin{equation}
(r,x,u_1, u_2) \mapsto   \sigma^\top f_0(r,x,u_1, u_2) 
\end{equation}
is  bounded on $[0,T]\times \R^d \times U_1 \times U_2$.
\end{enumerate}
\end{hypothesis}
Making use of convention \eqref{ystsol},
we recall that for given admissible controls $z_1 \in \shz_{1}(t)$, $z_2 \in \shz_{2}(t)$,
 $y(s) = y(s;t,x,z_1, z_2), s \in [t,T],$
will denote the unique solution of \eqref{state1}, whenever it exists.

\begin{rem} \label{PPolgrowth}
  \begin{enumerate}
    \item By definition of quasi-strong solution of \eqref{HJB1} (resp. \eqref{HJB2})
$$h(s,y(s)) = - H(s, y(s), \partial_x v(s, y(s))),$$ 
fulfills \eqref{hCond}. 
 \item $\partial_x v $ with polynomial growth implies that $v$ has polynomial growth. This in turn implies that $g$ has polynomial growth.
  \end{enumerate}
\end{rem}

Similarly to Lemma 4.10 of \cite{rg2} we state the fundamental lemma,
which is based on Theorem \ref{representation}.
The Hamiltonian $H_{CV}$ was introduced in
\eqref{Hext}.

\begin{lem}\label{FundamLemma1}
  We suppose Hypotheses \ref{3.1},
  \ref{hyp:ex-uniq} and \ref{Path0}.
  We set $H = H^{+}$ (resp. $H = H^{-}$)
  fulfilling Hypothesis \ref{hyp:BorelH}.
   Suppose the existence of 
  functions  $v$ satisfying Hypothesis \ref{3.7}
with respect to PDE \eqref{HJB1} (resp. \eqref{HJB2}).

Then $\forall (t,x) \in [0,T] \times \R^d$ and $\forall z_1 \in \shz_{1}(t)$, $\forall z_2 \in \shz_{2}(t)$, we have
  \begin{align*}
\tilde{J}(t,x; z_1, z_2) = v(t,x) + \int_t^T  \left(H_{CV}(r,y(r), \partial_x v(r,y(r)), z_1(r,y(r)), z_2(r,y(r)))  \right. & \\
   \left. - H(r,y(r),\partial_x v(r,y(r))) \right) dr + M_T, &
\end{align*}
where $M$ is a (square integrable) mean-zero r.v.
and $\tilde{J}$ is defined in \eqref{TJ0}.
\end{lem}
\begin{rem} \label{RTildeJ}
  We recall that, for a generic couple  $(z_1, z_2)$,
  $\tilde{J}(t,x;z_1,z_2)$ is a.s. finite though it could not be integrable.
  \end{rem}

\begin{proof} [Proof of Lemma \ref{FundamLemma1}]

  It is not possible to use It\^{o}'s formula because $\partial_s v, \partial_{xx}v$ do not necessarily exist.
  To overcome this difficulty we use the representation Theorem
  \ref{representation} for $u = v$. By Hypothesis \ref{3.7} we know that
  $v$ is a quasi-strong solution of
\begin{align*}
\shl v(t,x)&= -H(t,x,\partial_x v(t,x)), \\
 v(T,x)&= g(x).
\end{align*}
 By Proposition \ref{P44}, 
 the process $y$ in the statement fulfills the moments inequality \eqref{moments3} and $\partial_x v$ has polynomial growth, so that  \eqref{L21} is verified for $u=v$.

We wish to apply Theorem \ref{representation}.
We set $S_r = y(r)$ and
$F_0(\omega,r)  = f_0(r,y(r),z_1(r,y(r)),z_2(r,y(r))$.
We remark that if item 1. 
(resp. item 2.) of
Hypothesis \ref{3.7} is verified then
item 1. (resp. item 2.),
of Theorem \ref{representation}
is fulfilled.
This implies
\begin{equation}\label{Eq1}
g(y(T))= v(t,x)  +  \int_t^T \partial_x v(r,y(r)) \sigma(r,y(r)) dW_r + \shb^S(v)_T,
\end{equation}
where 
\begin{eqnarray}
  \shb^S(v)_s &=& \int_t^s  -H(r,y(r),\partial_x v(r,y(r)))dr \\ &+&
  \int_t^s  \partial_x v(r,y(r))  f(r,y(r),z_1(r,y(r)),z_2(r,y(r)))  dr, 
            \ s \in [t,T]. \nonumber
\end{eqnarray}
Now, $\int_t^T l(r,y(r;t,x,z_1,z_2), z_1(r,y(r)), z_2(r,y(r))) dr$ is a.s. finite by Hypothesis \ref{Path0} so we can add it to both sides of the equality \eqref{Eq1}
\begin{align*}\label{eqq2}
& \int_t^T l(r,y(r), z_1(r,y(r)), z_2(r,y(r))) dr + g (y(T))  =  v(t,x) \\
 &\qquad   + \int_t^T  \left[ -H(r,y(r),\partial_x v(r,y(r))) \right. \\
&\qquad \qquad \left. + H_{CV}(r,y(r),\partial_x v (r,y(r)), z_1(r,y(r)), z_2(r,y(r))) \right]   dr\\
&\qquad  + \int_t^T \partial_x v(r,y(r)) \, \sigma(r,y(r)) dW_r.
\end{align*}
Cauchy-Schwarz inequality together with the moments Property \ref{moments3}  imply that
\begin{align*}
\E \int_t^T |(\partial_x v(r,y(r))\, \sigma(r,y(r))|^2 dr < \infty, 
\end{align*}
so the stochastic integral process
$$ M_s:=  \int_t^s \partial_x v(r,y(r)) \, \sigma(r,y(r)) dW_r, \ s \ge t,
$$
is a square integrable martingale.
By the definition of $\tilde{J}$ in $\eqref{TJ0}$, the conclusion follows.
\end{proof}

 \subsection{Verification theorem and value of the game}

\label{S33}
 
The main result of this section is Theorem \ref{Verification3}.
We introduce below one hypothesis which can be verified
by a couple of functions
 $ z^\star_i : [0,T]  \times \R^d \times \R^d \ra U_i, $ for $i\in \{1,2 \}$.
 This notation constitutes a small (practical) abuse of notation, since
 those letters were indicating feedback controls.

\begin{hypothesis}\label{hyp_select}
  For all $(s,x,p) \in [0,T] \times \R^d \times \R^d$ we have
\begin{equation}
\left\{
\begin{array}{l l}\label{if_2}
&  \sup_{u_1\in U_1}   H_{CV} (s,x,p, u_1 , z^\star_2(s,x,p)) \\
&\qquad \qquad =\inf_{u_2\in U_2} \sup_{u_1\in U_1} H_{CV} (s,x,p, u_1, u_2),\\
&\inf_{u_2\in U_2} H_{CV} (s,x,p,z^\star_1(s,x,p), u_2)\\
&\qquad \qquad  = \sup_{u_1\in U_1}  \inf_{u_2\in U_2}   H_{CV} (s,x,p, u_1, u_2).
\end{array}
\right.
\end{equation}
\end{hypothesis}
\begin{rem} \label{rmk:minmax}
  \begin{enumerate}
  \item For a couple $(z^\star_1, z^\star_2)$,
    Hypothesis \ref{hyp_select} is equal to
    (2.26) and (2.27) in  \cite{FriedGames}.
    \item
  Alternative formulations to Hypothesis \ref{hyp_select}
  in the case of non-zero sum games
  are for example (2.10) in \cite{FriedGames} 
 and Assumption (A3) in \cite{Hamadene3}.
\end{enumerate}
\end{rem}

Proposition \ref{equiv}
 below  states the existence of such a couple
$(z^\star_1,z^\star_2)$ fulfilling Hypothesis \ref{hyp_select} if
 the Hamiltonian is continuous with respect to the control variables.
That proposition also states
an important saddle point property of the Hamiltonian.

\begin{prop}\label{equiv}
  \begin{enumerate}
    \item
  Assume $H_{CV} (s,x,p, u_1, u_2)$ to be  continuous in $(u_1,u_2)$.
     Then, there exists a couple $( z^\star_1, z^\star_2)$
such that
  Hypothesis \ref{hyp_select} is fulfilled.
\item Let be $  z^\star_i : [0,T]  \times \R^d \times \R^d \ra U_i $
    for $i\in \{1,2 \}$.
   The following are equivalent.
\begin{enumerate}
\item  The Isaacs' condition (see Definition \ref{Isaacs0}) holds 
and $( z^\star_1, z^\star_2)$
fulfills
  Hypothesis \ref{hyp_select}.
  \item
    For all $(s,x,p) \in [0,T] \times \R^d \times \R^d$ we have, for any $u_1 \in U_1$ and $u_2 \in U_2$,
\begin{equation}
\left\{
\begin{array}{l l}\label{if_3}
&H_{CV} (s,x,p, u_1 ,  z^\star_2(s,x,p)) \leq H_{CV} (s,x,p,  z^\star_1(s,x,p) ,  z^\star_2(s,x,p)) \\
& H_{CV} (s,x,p,  z^\star_1(s,x,p), u_2) \geq H_{CV} (s,x,p,  z^\star_1(s,x,p) ,  z^\star_2(s,x,p)).
\end{array}
\right.
\end{equation}
\end{enumerate}
\end{enumerate}
\end{prop}

\begin{proof}

  \begin{enumerate}
  \item We first prove the existence of a Borel function $  z^\star_2$
    such that
\begin{equation} \label{eq:z2}
\sup_{u_1\in U_1}   H_{CV} (s,x,p, u_1 ,  z^\star_2(s,x,p))  =\inf_{u_2\in U_2} \sup_{u_1\in U_1} H_{CV} (s,x,p, u_1, u_2).
\end{equation}
In order to do so, we apply Lemma 1 of \cite{benes} as follows.
Let $M= [0,T]\times \R^d\times \R^d$, $U= U_2$ and $A= \R$, and consider the function $k:M\times U\ra A$ defined as  
$$k(s,x,p,u) := \sup_{u_1\in U_1} H_{CV} (s,x,p, u_1, u).$$
Since $H_{CV}$ is Borel and continuous in $u_1$, as we mentioned earlier,
$k$ is a Borel function. On the other hand since
$H_{CV}$ is continuous in $(u_1,u)$ and the fact that $U_1$ is compact
it is easy to show that $k$ is continuous in $u$.
Define the function 
$$a(s,x,p):= \inf_{u_2\in U_2} \sup_{u_1\in U_1} H_{CV} (s,x,p, u_1, u_2))=  \inf_{u_2\in U_2} k(s,x,p,u_2).$$
Then $a(s,x,p)\in k(s,x,p,U)$.
The hypotheses of Lemma 1 in \cite{benes} are satisfied, hence there exists a Borel function $ z_2^\star : M \ra U$ such that
$$a(s,x,p) = k(s,x,p,  z_2^\star(s,x,p))$$ and the claim \eqref{eq:z2} is proved.

The existence of a Borel function $  z^\star_1(s,x,p)$ such that
\begin{equation*}
\inf_{u_2\in U_2} H_{CV} (s,x,p,   z^\star_1(s,x,p), u_2)  = \sup_{u_1\in U_1}  \inf_{u_2\in U_2}   H_{CV} (s,x,p, u_1, u_2)
\end{equation*}
follows using the same argument.

\item 
We apply the same argument of \cite{lucchetti}, Theorem 4.1.1.
We first prove that $(a) \Rightarrow (b)$.
We have
\begin{align}
&\inf_{u_2\in U_2} \sup_{u_1\in U_1} H_{CV} (s,x,p, u_1, u_2)  \nonumber \\
&\qquad = \sup_{u_1\in U_1} H_{CV} (s,x,p, u_1,   z^\star_2(s,x,p)) \label{eq_315}\\
&\qquad \geq H_{CV} (s,x,p,   z^\star_1(s,x,p),  z^\star_2(s,x,p))
\nonumber
  \\
& \qquad \geq  \inf_{u_2\in U_2} H_{CV} (s,x,p,   z^\star_1(s,x,p), u_2)
  \nonumber \\ 
& \qquad  = \sup_{u_1\in U_1}  \inf_{u_2\in U_2}   H_{CV} (s,x,p, u_1, u_2),\label{eq_316}
\end{align}
where
the equalities
\eqref{eq_315} and \eqref{eq_316} are justified by \eqref{if_2}.
Since Isaacs' condition
is in force, the above
inequalities are indeed equalities and therefore
 (b) is proved.

Viceversa, suppose (b) holds. Then 
\begin{align} \label{chainI}
&\inf_{u_2\in U_2} \sup_{u_1\in U_1} H_{CV} (s,x,p, u_1, u_2) \nonumber  \\
&\qquad \leq \sup_{u_1\in U_1} H_{CV} (s,x,p, u_1,   z^\star_2(s,x,p)) \nonumber\\
&\qquad \le H_{CV} (s,x,p,   z^\star_1(s,x,p),  z^\star_2(s,x,p))\\
& \qquad \le  \inf_{u_2\in U_2} H_{CV} (s,x,p,   z^\star_1(s,x,p), u_2) \nonumber\\
& \qquad  \leq \sup_{u_1\in U_1}  \inf_{u_2\in U_2}   H_{CV} (s,x,p, u_1, u_2). \nonumber
\end{align}
Since, trivially,  $$\sup_{u_1\in U_1}  \inf_{u_2\in U_2}   H_{CV} (s,x,p, u_1, u_2) \leq \inf_{u_2\in U_2} \sup_{u_1\in U_1} H_{CV} (s,x,p, u_1, u_2),$$
then Isaacs's condition (see Definition \ref{Isaacs0}) holds.
Moreover the inequalities in \eqref{chainI} become equalities, so
that Hypothesis \ref{hyp_select} is verified for $z^\star_1, z^\star_2$.
Finally, this implies (a).
\end{enumerate}

\end{proof}

\begin{rem} \label{rmk:Hamadene}
  If $H_{CV} $ is continuous in the control variables a direct consequence of
  Proposition \ref{equiv} implies that Isaacs' condition
  is equivalent to the existence of $  z^\star_i, i = 1,2$
  fulfilling \eqref{if_3}.

  This is similar to what is  stated in \cite{Hamadene1}, 
  above Lemma 2.1.
  \end{rem}
 
\begin{rem} \label{R-+}
  If Isaacs' condition
  holds,
  then $H^{-} \equiv H^{+}$ so that \eqref{HJB1} and
  \eqref{HJB2} coincide.
    \end{rem}

 Below, we state the most significant results of the section,
  i.e. the verification Theorem \ref{Verification3} and
  Corollary \ref{cor_ver_3}, which is devoted to the particular case
  where the diffusion is non-degenerate and  $H_{CV}$
  is continuous with respect to the control variables.
Indeed Theorem \ref{Verification3} and Corollary \ref{CVerification3}
link quasi-strong solutions of the BI equations with the
values of the game. Before stating the aforementioned theorem
we formulate a remark.

\medskip

\begin{rem} \label{RPayoff}
\begin{enumerate}
\item  The payoff function $J$,  defined in  \eqref{J0},
  is connected with $\tilde J$, defined in \eqref{TJ0}, by
 \begin{equation*}
  J(t,x; z_1,z_2) = \E(\tilde J(t,x; z_1,z_2)), \ t \in[0,T], x \in \R^d,  z_1 \in \shz_1(t), z_2 \in \shz_2(t),
  \end{equation*}
  provided previous expectation makes sense. We insist on the fact
that we
do not need any integrability assumption for every $z_1 \in \shz_1(t)$ and $z_2 \in \shz_2(t)$.
  \item  The upper and lower value $V^+, V^-$ for the stochastic differential game with payoff $J$ and state equation \eqref{state1} are defined in Definition \ref{Values} of Section \ref{Intro}. 
  is a solution to the SDE \eqref{state1}, see also \eqref{ystsol}.
 \end{enumerate}
\end{rem}

\begin{thm}\label{Verification3}
  Let $t\in [0,T], x \in \R^d$. Assume Isaac's condition (see Definition \ref{Isaacs0}) together with Hypotheses \ref{3.1},
\ref{hyp:ex-uniq},
  \ref{Path0} to
hold.
   We also assume  the existence of a function
   $v$  related to \eqref{HJB1} 
fulfilling  Hypothesis \ref{3.7}.
  
  Let $  z_i^*: [0,T] \times \R^d \times  \R^d \mapsto U_i$, for $i \in  \{ 1,2\}$, satisfying Hypothesis \ref{hyp_select} and set, for $i\in \{1,2 \}$,
 \begin{equation}\label{feed_con_p}
   z_i^{\star}(s, x) = z_i^*(s, x,\partial_x v(s,x)).
 \end{equation}
 Then the following properties hold.
 \begin{enumerate}
\item For any $(z_1,z_2) \in \shz_1(t) \times\shz_2(t)$, we have the following. 
\begin{enumerate}
\item  $ J(t,x;z^{\star}_1,z_2)$  is well-defined and greater than $-\infty$;
\item $ J(t,x;z_1,z^{\star}_2)$ is well-defined and smaller than $+\infty$.
\end{enumerate}
\item
  $J(t,x;z^{\star}_1,z^{\star}_2) =  v(t,x)$.
\item The couple $(z^{\star}_1,z^{\star}_2)$
  is a saddle point for the game, in the sense of
  Definition \ref{Saddle}.
\item The payoff functional evaluated at  $(z^\star_1,z^\star_2)$ is equal to both the upper and
  lower value of the game, i.e. $J(t,x;z^\star_1,z^\star_2) = V^+(t,x) = V^-(t,x)$.
  In particular the game admits a value.
\end{enumerate}
\end{thm}

    A direct consequence of Theorem \ref{Verification3}
    items 2. and 4. is the following. 
    \begin{cor} \label{CVerification3}
      Under the same assumptions of Theorem \ref{Verification3},
    for every $(t,x)$,
  we have
  $$ v(t,x) = V^-(t,x) = V^+(t,x).$$
  In particular $V := V^- = V^+$ is the unique quasi-strong
  solution of both \eqref{HJB1} and \eqref{HJB2}.
  \end{cor}

If the diffusion is non-degenerate and $H_{CV}$ is continuous with respect to the controls, the statement of Theorem \ref{Verification3} translates
into Corollary \ref{cor_ver_3} below.
In this case, taking into account Proposition \ref{equiv},
it is not necessary to suppose
the existence of functions $  z^\star_i, i = 1,2$ fulfilling
Hypothesis \ref{hyp_select}.
\begin{cor}\label{cor_ver_3}
  Let $t\in [0,T]$
  and assume the following.
  \begin{enumerate}
  \item   Hypothesis \ref{3.1} holds.
  \item The non-degeneracy condition \eqref{not-deg} holds.
    \item Isaacs' condition holds
      and $H_{CV} (s,x,p, u_1, u_2)$ is continuous in $(u_1,u_2)$.
        \item The validity of Hypothesis \ref{Path0}.
                 \item The existence of a  function
      $v$ fulfilling  Hypothesis \ref{3.7}.
   
    \end{enumerate}
    Then, there exists a couple $(z_1^\star, z_2^\star) \in \shz_1(t) \times \shz_2(t)$, 
    for which the same conclusions 1.--4. of Theorem \ref{Verification3} are true.
  \end{cor}
  \begin{proof}
   Proposition \ref{baseHypFeed} ensures the validity of Hypothesis
    \ref{hyp:ex-uniq}.
  By item 1. of Proposition \ref{equiv}, there exist Borel functions
  $  z_i^*: [0,T] \times \R^d \times  \R^d  \rightarrow
  U_i$, for $i \in  \{ 1,2\}$ verifying Hypothesis \ref{hyp_select}.
 Hence, the assumptions of Theorem \ref{Verification3} are satisfied.
\end{proof}
\begin{rem} \label{R326}
  Theorem \ref{Verification3} applies of course quite
  generally when $\sigma$ is non-degenerate.
  Nevertheless it
  can also be used (or adapted) in some cases, even
  when $\sigma$ is  degenerate.
We explain this below.
\begin{enumerate}
\item
  Let us fix $d = 1$ for simplicity.
  We suppose that
  $f(s,x,u_1,u_2) = \sigma^2(x) f_0(s,x,u_1,u_2)$,
  where in particular $\sigma$, defined in \eqref{coef_space}, does not depend on time, it only vanishes in a point
  and $f_0: [0,T] \times \R \times U_1 \times U_2 \rightarrow \R$
  is such that $\sigma f_0$ is  bounded. We also suppose that $\frac{1}{\sigma}$ is not integrable at infinity.
  Then Hypothesis \ref{3.1} is fulfilled, because of Proposition \ref{PLampertiZvonkin} below.
\item
  In a fully general degenerate case,
    if $f$ is uniformly Lipschitz in $(x,u_1,u_2)$, one could
    still establish a verification theorem similar to Theorem \ref{Verification3},
    taking into account the following modifications.
\begin{itemize}
\item    Hypothesis \ref{hyp:ex-uniq} is not  fulfilled in general.
  In this case we restrict (for $i = 1,2$),  $\shz_i(t)$ to the class
$\sha_i(t)$   of Lipschitz feedback control functions $z_i$.
  Now, item 2. of Proposition \ref{baseHypFeed} for every $
  z_i \in \sha_i(t), i = 1,2$, is verified.
  In particular, the (modified) Hypothesis \ref{hyp:ex-uniq} is fulfilled
  if we replace $\shz_i(t)$
  with $\sha_i(t)$.
\item We also need 
  the validity of
  item 1. of Hypothesis \ref{3.7}. 
  A typical example, where this holds true, can be formulated making use of Remark 4.2 of \cite{CR1}:
  in particular we require that $x \mapsto \sigma(t,x),
  x \mapsto H(t,x,\partial_x u(t,x))$ and $g$ are of class $C^1$,
  for every $t \in [0,T]$, $\partial_x \sigma$ is bounded
  and $\partial_xg, \partial_xH(t,x,\partial_x u(t,x))$
  have polynomial growth uniformly with respect to $t$.
\end{itemize}

\end{enumerate}
\end{rem}
\begin{prop} \label{PLampertiZvonkin}
Let 
$b: [0,T] \times \R$
be a bounded Borel function, $x_0 \in \R$, 
and $\sigma: \R \rightarrow \R$ be
  a function of class $C^1$ with bounded derivative,  such that the following holds. 
\begin{enumerate}
\item $\sigma(x_0) = 0$;
\item $\sigma(x) \neq 0 $ for any $x \neq x_0$;
\item $\sigma b$ is bounded;
\item $$\int_{x_0 +1}^\infty \frac{1}{\sigma(y)}dy = \int_{-\infty}^{x_0 -1}\frac{1}{\sigma(y)}dy = + \infty.$$
\end{enumerate}
Let $t \in [0,T[$.
Then,
for any $x_1 \in \R$,
the SDE
\begin{equation}\label{SDE1}
\left\{
\begin{array}{ll}
 dX_s =   \sigma^2(X_s)b(s,X_s) ds +\sigma(X_s) dW_s,\\
 X_t = x_1,
\end{array}
\right.
\end{equation}
admits strong existence and pathwise uniqueness.
\end{prop}
\begin{proof}
  We fix $t=0$, for simplicity.
  Existence in law is a consequence of Girsanov's theorem, setting $\tilde W_s = W_s + \int_0^s b(r,X_r)\sigma(X_r) dr$. Uniqueness in law also follows by Girsanov.

Let us discuss pathwise uniqueness. Without losing the generality, assume $x_0 = 0$. Suppose first $x_1 = 0$.
Then $X_t \equiv 0$ is a solution. Given any other solution, by uniqueness in law, its law has to be concentrated at $x_0=0$.
So pathwise uniqueness follows.
Suppose $x_1 \neq 0$, for instance $x_1 >0$. Define 
\begin{equation}
F(x) = \int_{x_1}^x \frac{1}{\sigma(y)} dy,\qquad x>0.
\end{equation}
By hypothesis,  $ F: ]0,\infty[ \mapsto \R$ is a $C^1$-diffeomorphism, with   $F'(x) = \frac{1}{\sigma(x)}$ and $F''(x) = -\frac{\sigma'}{\sigma^2}(x)$. 
Moreover, since $\sigma$ is Lipschitz, for every $x>0$, it  holds $\sigma(x) = |\sigma(x)- \sigma(0)| \leq C x$, for some constant $C$.
So $\int_0^\eta \frac{1}{\sigma(y)} dy = + \infty$ for any $\eta >0$.
By It\^o's formula, setting $Y_s= F(X_s)$, we get
$$ Y_s  = W_s  +\int_0^s  (b \sigma - \sigma')(r,F^{-1}(Y_r)) dr, $$
at least until $X$ reaches $x_0 = 0$.
However, this will a.s. never happen. Otherwise, let $\tau$ be
the first  hitting time of $X$ at zero;
then $\lim_{s\ra \tau-} Y_s = +\infty$.
Now,
$Y_s$ admits pathwise uniqueness and strong existence by  \cite{veretennikov1982}, Theorem 6.
Hence, we cannot have $\lim_{t\ra \tau} Y_s = + \infty$. At this point, $X$ lives in $]0,\infty[$ and is the unique solution. 

\end{proof}

\begin{proof}[Proof of Theorem \ref{Verification3}]
 Let $v$ be a solution to the lower Bellman-Isaacs equation \eqref{HJB1}, or equivalently \eqref{HJB2}.
 Applying Lemma \ref{FundamLemma1} to a generic couple
$(z_1,z_2) \in \shz_1(t) \times \shz_2(t)$,
borrowing the notation of \eqref{ystsol}, i.e. $y(r)=y(r;t,x,z_1,z_2)$,
we have
\begin{align}
 &\tilde{J}(t,x; z_1,z_2) =  v(t,x)  + M_T(z_1,z_2)  \label{2.3_0} \\
  &\qquad + \int_t^T \left(H_{CV}(r,y(r),
                            \partial_x v(r,y(r)), z_1(r,y(r)),z_2(r,y(r))) \right.  \nonumber \\
& \qquad \qquad  \left. - \sup_{u_1 \in U_1} \inf_{u_2\in U_2} H_{CV}(r,y(r),\partial_x v(r,y(r)), u_1,u_2) \right) dr, \nonumber 
\end{align}
where $M_T(z_1,z_2)$ is a mean-zero  square integrable r.v.

We apply \eqref{2.3_0} to $(z^{\star}_1,z_2)$ so that, for $y(r)=y(r;t,x,z^\star_1,z_2),$
\begin{align} 
&\tilde{J}(t,x; z^{\star}_1,z_2)  = v(t,x)  +   M_T(z_1^{\star}, z_2) \nonumber \\
& \qquad   + \int_t^T \left( H_{CV}(r,y(r), \partial_x v(r,y(r)), z^{\star}_1(r,y(r)),z_2(r,y(r)))  \right. \nonumber \\
  &\qquad  \left. - \sup_{u_1 \in U_1} \inf_{u_2\in U_2} H_{CV}(r,y(r),\partial_x v(r,y(r)), u_1,u_2) \right) dr   \label{eq_2_3_4} \\
  & \qquad \geq v(t,x)+    M_T(z_1^{\star}, z_2) \nonumber \\
&  \qquad  + \int_t^T \inf_{u_2\in U_2} \left( H_{CV}(r,y(r), \partial_x v(r,y(r)), z^{\star}_1(r,y(r)),u_2)   \right.
  \nonumber \\
  &\qquad  \left. - \sup_{u_1 \in U_1} \inf_{u_2\in U_2} H_{CV}(r,y(r),\partial_x v(r,y(r)), u_1,u_2) \right) dr. \label{eq_2_3_5}
\end{align}
By the second equality of \eqref{if_2} in Hypothesis \ref{hyp_select} related to
$(z^\star_1, z^\star_2)$,
the term inside the integral of \eqref{eq_2_3_5} vanishes.
Therefore
\begin{equation}\label{eq_2_3_3}
\tilde{J}(t,x; z^{\star}_1,z_2)  \geq  v(t,x) + M_T(z_1^{\star}, z_2).
\end{equation}
It follows that $\tilde{J}(t,x, z^{\star}_1,z_2) $ is quasiintegrable, i.e.
its expectation is well-defined: in particular it belongs to
$]-\infty, +\infty]$ since the other terms are integrable.
This obviously implies item 1.(a).
Taking expectation of \eqref{eq_2_3_3}
\begin{equation}\label{eq_2_3_3a}
J(t,x; z^{\star}_1,z_2)   \geq  v(t,x).
\end{equation}
By analogue arguments we can establish the proof of item 1.(b)
and 
\begin{equation}\label{eq_2_3_5a}
  J(t,x; z_1, z^{\star}_2)  \leq v(t,x).
\end{equation}
Inequalities \eqref{eq_2_3_3a} and \eqref{eq_2_3_5a} prove also that 
\begin{equation}\label{eq_2_3_5b}
  J(t,x; z^{\star}_1, z^{\star}_2)  = v(t,x),
\end{equation}
which
proves item 2.

Moreover, by \eqref{eq_2_3_5b}, \eqref{eq_2_3_3a} and \eqref{eq_2_3_5a},
we also have
\begin{equation}\label{eq_2_3_5e}
   J(t,x; z_1, z^\star_2)  \leq J(t,x; z^\star_1, z^{\star}_2)  \leq   J(t,x; z^\star_1, z_2)
\end{equation}
and item 3. is proved.
Item 4. is a direct consequence of Proposition
\ref{prop:verification} below.

\end{proof}

\begin{rem} \label{RVerifHamadene}
  In the case of two players, the conclusion of the verification theorem \cite{Hamadene3}
  is the same as ours. Nevertheless our assumptions are different.
\begin{enumerate}
\item The assumptions of \cite{Hamadene3} imply that
  $J(t,x, z_1,z_2)$ is well-defined for every $z_1 \in \shz_1(t), z_2 \in \shz_2(t)$.
  In our case, it could not be always the case.
\item
  In the hypothesis of Theorem \ref{Verification3} and
  Corollary \ref{cor_ver_3} we do not necessarily assume  the following additional
hypothesis that appear in \cite{Hamadene3}:
$ p \mapsto H_{CV}(s,x,p, z^\star_1(s,x,p) ,z^\star_2(s,x,p))$ is continuous for any fixed $(s, x)$, see \cite{Hamadene3}, Assumption A3, item (ii).
\item Our context includes the possibility for  $\sigma$ to be degenerated.
\end{enumerate}

Our methodology, alternatively, proves the existence of Nash equilibrium
supposing the existence of $C^{0,1}$-quasi-strong solutions
of the Bellman-Isaacs PDEs.
\end{rem}
\begin{prop} \label{prop:verification}
  Let $t\in [0,T], x \in \R^d$.
 Suppose the existence of a couple $(z^{\star}_1,z^{\star}_2)$ which 
is a saddle point for the game.

Then the payoff functional evaluated at  $(z^\star_1,z^\star_2)$ is equal to both the upper and
  lower value of the game, i.e. $J(t,x;z^\star_1,z^\star_2) = V^+(t,x) = V^-(t,x)$.
  In particular the game admits a value.

  \end{prop}

  \begin{proof}
We recall that, by Definition \ref{Values},
$$V^-(t,x) = \sup_{z_1\in \shz_1(t)} \inf_{z_2\in \shz_2(t)} J^-(t,x; z_1,z_2)$$
and observe that, trivially, we have
\begin{equation}\label{item4_1}
 \inf_{z_2\in \shz_2(t)} \sup_{z_1\in \shz_1(t)} J^-(t,x; z_1,z_2)   \leq \sup_{z_1\in \shz_1(t)} J^-(t,x; z_1,z_2^\star).
\end{equation}
By Remark \ref{RSaddle}
we have
\begin{equation}\label{item4_3a}
  \sup_{z_1\in \shz_1(t)} J^-(t,x; z_1,z_2^\star)   \le J(t,x; z^{\star}_1,z^{\star}_2)
  \le \inf_{z_2\in \shz_2(t)} J^-(t,x; z_1^\star,z_2).
\end{equation}
  Therefore, using \eqref{item4_1} and \eqref{item4_3a}, we obtain
\begin{equation}\label{item4_4}
  \inf_{z_2\in \shz_2(t)} \sup_{z_1\in \shz_1} J^-(t,x; z_1,z_2) 
  \le J(t,x; z^{\star}_1,z^{\star}_2) \le
   \inf_{z_2\in \shz_2(t)} J^-(t,x; z_1^\star,z_2)
 \end{equation}
and  trivially we have
\begin{equation}\label{item4_6}
 \inf_{z_2\in \shz_2(t)} J^-(t,x; z_1^\star,z_2) \leq \sup_{z_1\in \shz_1(t)}  \inf_{z_2\in \shz_2(t)}  J^-(t,x; z_1,z_2).
\end{equation}
By \eqref{item4_4}
and \eqref{item4_6} we obtain 
\begin{equation}\label{item4_7}
  \inf_{z_2\in \shz_2(t)} \sup_{z_1\in \shz_1(t)} J^-(t,x; z_1,z_2)  
   \le J(t,x; z^{\star}_1,z^{\star}_2) \le
    \sup_{z_1\in \shz_1(t)}  \inf_{z_2\in \shz_2(t)}  J^-(t,x; z_1,z_2).
  \end{equation}
Since
\begin{equation*}
 \sup_{z_1\in \shz_1(t)}  \inf_{z_2\in \shz_2(t)}  J^-(t,x; z_1,z_2) \leq \inf_{z_2\in \shz_2(t)} \sup_{z_1\in \shz_1(t)} J^-(t,x; z_1,z_2),
\end{equation*}
we have proved
\begin{equation}\label{Esupinf-}
  \sup_{z_1\in \shz_1(t)}  \inf_{z_2\in \shz_2(t)}  J^-(t,x; z_1,z_2) = \inf_{z_2\in \shz_2(t)} \sup_{z_1\in \shz_1(t)} J^-(t,x; z_1,z_2)
  = J(t,x; z^*_1,z^*_2).
\end{equation}
A similar argument allows to establish
\begin{equation}\label{Esupinf+}
  \sup_{z_1\in \shz_1(t)}  \inf_{z_2\in \shz_2(t)}  J^+(t,x; z_1,z_2) = \inf_{z_2\in \shz_2(t)} \sup_{z_1\in \shz_1(t)} J^+(t,x; z_1,z_2)
  = J(t,x; z^*_1,z^*_2).
\end{equation}
Hence we have proved that
$$ V^-(t,x) = J(t,x, z^{\star}_1,z^{\star}_2) = V^+(t,x),$$
which concludes the proof of the proposition.
\end{proof}

\section{The case of control theory}

\label{S5}

As mentioned at the end of the Introduction, the techniques
described in previous sections allow also to establish
a verification theorem for stochastic control problem,
which generalizes Theorem 4.9 of \cite{rg1}.
This is the object of Theorem \ref{verificationC},
which follows from Lemma \ref{FundamLemmaC}.
That corollary is a consequence of a result very close to Lemma \ref{FundamLemma1}
and practically extends
Lemma 4.10 in \cite{rg1}.

We start with the basic assumptions on the coefficients.
\begin{hypothesis}\label{3.1Control}
  As in \eqref{coef_space}, $\sigma$ and $f$ are Borel functions.
  Remark that $f:[0,T]\times\R^d\times U \rightarrow \R^d$ depends here only on a single control.
  

 Moreover there exists $K > 0$ such that $\forall x, x_1, x_2 \in \R^d, t \in [0,T]$, 
  $\forall u \in U$,  the properties below hold.
  \begin{enumerate}
    \item $f$ is continuous in $x$ for every $(t,u)$.
\item  $ \langle f(t,x_1,u)- f(t,x_2,u), x_1-x_2 \rangle + ||\sigma(t,x_1)-\sigma (t,x_2)||^2 \leq K|x_1-x_2|^2 $.
\item $ \vert f(t,x,u) \vert + \Vert \sigma(t,x) \Vert  \leq K (1+|x|)$.
\end{enumerate}      
\end{hypothesis}

We now describe our optimal control problem.
Similarly to  Section \ref{Intro},
let us fix a stochastic basis
$(\Omega, \shf, (\shf_s) _{s\in [0,T]}, \P) $ satisfying the usual conditions, a finite dimensional Hilbert space,
say $\R^d$ that will be the state space, a finite dimensional Hilbert space, say $\R^m$ (the noise space), one compact set
$U \subseteq \R^k$ (the control space). We will deal with a fixed horizon problem so that we fix $T\in ]0, \infty[$ at the beginning.
$W$ is a
$(\shf_s)_{s\in [0,T]}$-$d$-dimensional Brownian motion.

Given an initial time and state $(t,x)\in [0,T]\times \R^d$,
the state equation is
\begin{equation}\label{stateC}
\left\{
\begin{array}{l l}
dy(s)= f(s,y(s),z(s)) ds + \sigma (s,y(s)) dW_s, \\
 y(t)= x.
\end{array}\right. 
\end{equation}
The process $\shz(t) \ni z :[t,T] \times \Omega \mapsto U$ is the control
processes, where $\shz(t)$ is the set of \textit{admissible control processes}, that is $(\shf_s)_{s\in [t,T]}$-progressively measurable processes taking values in $U$. We remark that, here, the set $\shz(t)$ is a set of processes and it is not anymore constituted by
functions defined on $[0,T] \times \R^d$.
Adopting an analogous formulation as for the game
theory part, given a function $\tilde z$ as ''feedback'' control,
$z(s) = \tilde z(s,y(s))$
would be a control in our sense.
Therefore, by a language abuse,
the class of admissible controls in this section is larger.
For this reason,  Hypothesis \ref{3.1Control}
 looks less general than Hypothesis \ref{3.1} formulated in the game theory setting.

By Theorem 1.2 in \cite{KR}, for every $z \in \shz(t)$,
there is a unique solution
 to \eqref{stateC} denoted by 
\begin{equation}\label{ystC}
y(s;t,x,z) \ \text{or} \ y(s), \ s\in [t,T], x\in \R^d, z\in \shz(t).
\end{equation}
We observe that, by 
Theorem 4.6 of \cite{KR}, $y$ fulfills
a moments inequality practically identical to the one
in Property \ref{moments3}.

The second hypothesis, concerning the running and terminal costs, is an adaptation
of Hypothesis \ref{Path0}.

\begin{hypothesis}\label{Path0Control}
  Let $l: [0,T]\times \R^d \times U \rightarrow \R, \
  g : \R^d \rightarrow \R $ be Borel functions. Moreover, for any
$t \in [0,T]$ we assume
  that, for all $z \in \shz(t)$, the function $s \mapsto l(s, y(s), z(s))$ is integrable in $[t,T]$, $\omega$ a.s. In particular $ -\infty  < \int_t^T l(s, y(s), z(s)) ds  < \infty$ a.s.
\end{hypothesis}

The payoff function  is defined as
\begin{equation}\label{JC}
 J(t,x;z)= \E(\tilde{J}(t,x;z)),
\end{equation}
provided previous expectation exists (otherwise it will be set to $+\infty$),
where 
\begin{equation}\label{TJC}
 \tilde{J}(t,x;z)=  \int_t^T l(s,y(s;t,x,z), z(s)) ds + g (y(T;t, x, z)),
\end{equation}
adopting very close notations to \eqref{J0} and below in Section \ref{Intro}
and in Hypothesis \ref{Path0} and Definition \ref{D14} in Section \ref{Section2}.
The objective is to minimize the payoff, hence the value function is
\begin{equation} \label{eq:ValueContr}
V(t,x) = \inf_{z \in \shz(t)}  J(t,x;z).
\end{equation}

\begin{dfn}\label{OptimalC}
Let $t\in [0,T]$. If there exists a control $z^\star \in \shz(t)$ such that $J(t,x; z^\star) = V(t,x)$ for any $x\in \R^d$, we say that the control $z^\star$ is optimal for the problem \eqref{stateC} and \eqref{JC}. 
\end{dfn}

The current value Hamiltonian is defined,
 for $(s,x,p,u) \in [0,T] \times \R^d \times \R^d \times U$, as
\begin{equation*}
  H_{CV}(s,x,p,u) = \langle f(s,x,u),p\rangle + l(s,x,u),
\end{equation*}
and the minimum value Hamiltonian (supposed to be Borel)
is
\begin{equation}\label{HamC}
H(s,x,p)= \inf_{u\in U} H_{CV}(s,x,p,u). 
\end{equation}
Defining formally the operator $\shl$ as
\begin{equation*}
\shl u(s,x) = \partial_s u (s,x)
+ \frac{1}{2} Tr [\sigma^\top(s,x) \partial_{xx}u(s,x) \sigma(s,x)],
\end{equation*}
it is possible to write the HJB equation
associated with 
problem \eqref{stateC} and \eqref{JC} as 
\begin{equation}\label{HJBC}
\left\{
\begin{array}{l l}
 \shl v (s,x) + H (s,x,\partial_x v(s,x)) =0, \\
v(T,x)= g(x).
\end{array}\right.
\end{equation}
We will consider 
 quasi-strong and quasi-strict solutions for the HJB equation as in Definitions \ref{strong1} and \ref{strict1}.

 We remark that, 
 in the game theory setting, we were using similar notations on the coefficients but there
 $f$ and $l$ were functions of  four variables  $(s,x,u_1,u_2)\in [t,T]\times \R^d \times U_1 \times U_2$,
while here they only depend  on  three variables   $(s,x,u)\in [t,T]\times \R^d \times U$.

The proof of the lemma  below
can be done following exactly the same lines as those
of Lemma \ref{FundamLemma1}.
\begin{lem}\label{FundamLemmaC}
  We assume Hypotheses  \ref{3.1Control} and \ref{Path0Control}. We also suppose the existence of a function $v$
  such that
  satisfying Hypothesis \ref{3.7},
with $U_1 \times U_2$ (resp. $(u_1,u_2)$) replaced by $U$ (resp. $u$),
   where \eqref{HJB1} (resp. \eqref{HJB2}) is replaced by \eqref{HJBC}.

   Then, $\forall (t,x) \in [0,T] \times \R^d$, $\forall z \in \shz(t)$, setting $y(r)= y(r;t,x,z)$, as in \eqref{ystC}, we have
\begin{eqnarray}\label{17}
\tilde{J}(t,x;z) &=& v(t,x)+   \int_t^T \left(H_{CV} (r,y(r), \partial_x v(r,y(r),z(r)) - H(r,y(r), \partial_x v(r,y(r))) \right)dr \nonumber \\
&+& M_T, 
\end{eqnarray}
where $M_T$ is a mean-zero (square-integrable) r.v.
and $\tilde{J}$ is defined in \eqref{TJC}.

\end{lem}

We now state a   more general verification theorem than Theorem 4.9 of \cite{rg1}.

\medskip

\begin{thm}\label{verificationC}
    We assume Hypotheses  \ref{3.1Control} and \ref{Path0Control}. 
    We also suppose the existence of a function $v$ satisfying Hypothesis \ref{3.7},
with $U_1 \times U_2$ (resp. $(u_1,u_2)$) replaced with $U$ (resp. $u$),
   where \eqref{HJB1} (resp. \eqref{HJB2}) is replaced by \eqref{HJBC}.
    
  Let $t\in [0,T], x \in \R^d$.  Then we have the following.
\begin{enumerate}
 \item  For any $z\in \shz(t)$ the functional \eqref{JC} is well-defined. In particular the random variable
   $\tilde{J}(t,x,z):  \Omega \mapsto \R $ defined in \eqref{TJC}  is
quasiintegrable and its expectation is strictly
  greater than $-\infty$.
 \item  $v(t,x)\leq V(t,x)$.
\item If $z^\star \in \shz(t)$ satisfies (setting $y(r) = y(r;t,x,z^\star)$, as in \eqref{ystC})
\begin{equation}\label{optC}
 H(r,y(r), \partial_x v(r,y(r))) = H_{CV} (r,y(r), \partial_x v(r,y(r)),z^\star(r)),
\end{equation}
for a.e. $(r,x) \in [t,T] \times \R^d$, $\P$-a.s., then $z^\star$ is optimal in the sense of Definition \ref{OptimalC}. Moreover $v(t,x)= V(t,x)$ and $V(t,x)$ is finite.
\end{enumerate}
\end{thm}
\begin{proof}
  Applying Lemma \ref{FundamLemmaC} for a $z \in \shz(t)$, we obtain 
\begin{align}\label{2.3_00}
  \tilde{J}(t,x, z) &= v(t,x) + \int_t^T \left(H_{CV}(r,y(r;t,x,z(r)), \partial_x v(r,y(r;t,x,z), z(r))  \right. \nonumber \\
                             &\qquad - \left. H(r,y(r;t,x,z),\partial_x v(r,y(r;t,x,z)) \right)dr  + M_T(z),
 \end{align}
where $M_T(z)$ is a mean-zero (square-integrable) r.v.

Obviously, the integral in  \eqref{2.3_00} is always greater or equal than zero, hence, 
\begin{align}\label{2.3_01}
\tilde{J}(t,x, z) &\geq v(t,x)  + M_T(z).
\end{align}
Taking the expectation, it follows that, for a generic $z \in \shz(t)$
$$ J(t,x;z) \geq v(t,x) $$
and therefore item 1. follows.
Then, taking the infimum over $z$ allows to prove item 2.

Concerning item 3., if $z^\star$ satisfies \eqref{optC}, the integral 
in \eqref{2.3_00} vanishes and so $v(t,x) = J(t,x;z^\star)$.
Consequently $v(t,x) = V(t,x)$ and the result is finally proved.
\end{proof}

\section*{Acknowledgments}

Both authors are grateful to  the Associate Editor, the Editor and
especially the Referees,
for stimulating them to drastically improve the previously submitted
versions.
The research of the second named author
was partially supported by the  ANR-22-CE40-0015-01 project (SDAIM).

\addcontentsline{toc}{chapter}{Bibliography}
\bibliographystyle{plain}
\bibliography{../../../../../BIBLIO_FILE/biblioCarlo}

\end{document}